\documentclass[10pt,a4,leqno]{amsart}
\usepackage{amsthm}
\usepackage{amsmath}
\usepackage{amssymb}
\usepackage{amsfonts}
\usepackage{mathrsfs}
\usepackage{graphicx}
\usepackage{bbm}
\usepackage[
colorlinks=true, citecolor=blue, linkcolor=blue, urlcolor=red]{hyperref}
\usepackage{color}
\usepackage[all]{xy}
\usepackage{comment}
\usepackage{here}
\usepackage{amscd}
\usepackage{caption}
\usepackage{listings}
\usepackage{setspace}
\usepackage{booktabs}
\allowdisplaybreaks

\newcommand{\mathsym}[1]{{}}
\newcommand{\unicode}[1]{{}}

\makeatletter

\@addtoreset{equation}{section}
\makeatother

\newtheoremstyle{my theoremstyle}
{1.0em}                    
    {1.0em}                    
    {\itshape}                   
    {}                           
    {\scshape}                   
    {.}                          
    {.5em}                       
    {}  

\newtheoremstyle{dfn}
{1.0em}                    
    {1.0em}                    
    {}                   
    {}                           
    {\scshape}                   
    {.}                          
    {.5em}                       
    {}  
    
\theoremstyle{my theoremstyle}
   \newtheorem{thm}{Theorem}[section]
   
   \newtheorem{prop}[thm]{Proposition}
   \newtheorem{cor}[thm]{Corollary}
   \newtheorem{conj}[thm]{Conjecture}
      
\theoremstyle{dfn}

\theoremstyle{remark}   
   
   \newtheorem{rmk}[thm]{{\scshape Remark}}

\newcommand{\Z}{\mathbb{Z}}
\newcommand{\Q}{\mathbb{Q}}

\newcommand{\C}{\mathbb{C}}
\newcommand{\F}{\mathbb{F}}
\newcommand{\Fp}{\overline{\mathbb{F}}_p}
\newcommand{\al}{\alpha}
\newcommand{\e}{\varepsilon}

\numberwithin{equation}{section}

\newcommand{\ua}{\underline{a}}
\newcommand{\cua}{\check{\underline{a}}}

\newcommand{\spec}{\operatorname{Spec}}
\newcommand{\df}{\mathscr{F}^{\rm Dw}}
\newcommand{\hB}{\widehat{B}}

\newcommand{\vp}{\varphi}
\newcommand{\wU}{\widehat{U}}
\newcommand{\wV}{\widehat{V}}

\makeatletter
\@namedef{subjclassname@2020}{%
  \textup{2020} Mathematics Subject Classification}
\makeatother

\date{\today}

\begin{document}
\title[Transformation formula of Dwork's $p$-adic hypergeometric function]{Transformation formula of Dwork's $p$-adic hypergeometric function}
\author{Yusuke Nemoto}
\date{\today}
\address{Graduate School of Science and Engineering, Chiba University, 
Yayoicho 1-33, Inage, Chiba, 263-8522 Japan.}
\email{y-nemoto@waseda.jp}

\maketitle

\begin{abstract}
In this paper, we give a transformation formula of Dwork's $p$-adic hypergeometric function between $t$ and $t^{-1}$. 
As an appendix, we introduce a finite analogue of this transformation formula, which implies the special case of the above transformation formula.  
\end{abstract}

\section{Introduction}
Let $p$ be a prime and $d \geq 0$ be an integer. 
Let $\ua=(a_0, \ldots, a_d) \in \Z_p^{d+1}$ (resp. $(b_1, \ldots, b_d) \in (\Z_p \setminus \Z_{ \leq 0})^d$) be a $(d+1)$-tuple (resp. $d$-tuple). 
We define the hypergeometric series by 
\begin{align*}
{_{d+1}F_{d}}\left( 
\begin{matrix}
a_0, \cdots, a_d \\
b_1, \ldots, b_d 
\end{matrix}
; t
\right)
=\sum_{n=0}^{\infty} \dfrac{(a_0)_n \cdots (a_d)_n}{(1)_n(b_1)_n \cdots (b_d)_n} t^n \in \Q_p[[t]], 
\end{align*}
where $(a)_n=a(a+1)\cdots (a+n-1)$ denotes the Pochhammer symbol. 
In this paper, we focus on a special case 
\begin{align}
F_{\ua}(t):={_{d+1}F_{d}}\left( 
\begin{matrix}
a_0, \cdots, a_d \\
1, \ldots, 1 
\end{matrix}
; t
\right), \label{HG}
\end{align}
which belongs to $\Z_p[[t]]$. 
Put $D=t\frac{d}{dt}$. 
Then \eqref{HG} becomes a solution of the hypergeometric differential equation 
$$P_{{\rm HG }, \ua} F_{\ua}(t) =0, \quad P_{{\rm HG}, \ua}=D^{d+1} - t(D+a_0) \cdots (D+a_d). $$
For $a \in \Z_p$, we define the Dwork prime $a'$ by $(a+l)/p$, where $l \in \{0, \ldots, p-1\}$ is the unique integer such that $a+l \equiv 0 \pmod{p}$. 
The $i$th Dwork prime is defined by $a^{(i)}=(a^{(i-1)})'$ and $a^{(0)}=a$. 
For $\ua=(a_0, \ldots, a_d) \in \Z_p^{d+1}$, $\ua'$ (resp. $\ua^{(i)}$) denotes $(a'_0, \ldots, a'_d)$ (resp. $(a_0^{(i)}, \ldots, a_d^{(i)})$).    
In the paper \cite{Dwork}, Dwork defines the $p$-adic hypergeometric function by 
$$\df_{\ua}(t)=\dfrac{F_{\ua}(t)}{F_{\ua'}(t^p)} \in \Z_p[[t]], $$
and proves the congruence relation (see Thm. 2 in loc.cit.)
$$\df_{\ua}(t) \equiv \dfrac{[F_{\ua}(t)]_{<p^n}}{[F_{\ua'}(t^p)]_{<p^n}} \pmod {p^n \Z_p[[t]]}, $$
where for a power series $f(t)=\sum_n a_n t^n$, $[f(t)]_{<m}$ denotes the truncated polynomial $\sum_{n<m} a_n t^n$.  
Thanks to this, $\df_{\ua}(t)$ becomes a $p$-adically analytic function in the sense of Krasner, i.e. an element of the Tate algebra as follows. 

For $f(t) \in \Z_p[t]$, $\overline{f(t)} \in \F_p[t]$ denotes the reduction of $f(t)$ modulo $p$. 
Let $N$ be a sufficiently large integer such that 
$$\{[\overline{F_{\ua^{(i)}}(t)}]_{<p}\}_{i \geq 0}=\{[\overline{F_{\ua^{(i)}}(t)}]_{<p}\}_{0 \leq i \leq N}$$ 
as subsets of $\F_p[t]$. 
Define a polynomial by 
\begin{align} 
h_{\ua}(t)=\prod_{i=1}^N [F_{\ua^{(i)}}(t)]_{<p}. \label{ha}
\end{align}
Then $\df_{\ua}(t)$ defines an element of the Tate algebra (cf. \cite[Corollary 2.3]{A3}) and we denote the same notation, i.e. 
$$\mathscr{F}_{\ua}^{\rm Dw}(t) \in \varprojlim_{n} \left(\Z_p/p^n\Z_p[t, h_{\ua}(t)^{-1}] \right). $$

The hypergeometric functions over $\C$ have many transformation and summation formulas (cf. \cite{Slater}). However, little is known about these for Dwork's $p$-adic hypergeometric functions. 
In the paper \cite{Wang, Wang2}, Wang conjectures the transformation formula of Dwork's $p$-adic hypergeometric function between $t$ and $t^{-1}$ as follows. 
Put 
$$\Z_p\langle t, t^{-1}, h_{\ua}(t)^{-1}\rangle=\varprojlim_{n} \left(\Z_p/p^n\Z_p[t, t^{-1}, h_{\ua}(t)^{-1}] \right). $$
By \cite[Proposition 4.11 (3)]{Wang}, we have $\Z_p\langle t, t^{-1}, h_{\ua}(t^{-1})^{-1}\rangle=\Z_p\langle t, t^{-1}, h_{\ua}(t)^{-1}\rangle$, hence there is an involution 
$$\iota \colon \Z_p\langle t, t^{-1}, h_{\ua}(t)^{-1}\rangle \to \Z_p\langle t, t^{-1}, h_{\ua}(t)^{-1}\rangle; \quad f(t) \mapsto f(t^{-1}). $$

\begin{conj}[{\cite[Conjecture 4.18]{Wang}}, {\cite[Conjecture 4.7]{Wang2}}] \label{conj}
For any $a \in \Z_p$, 
let $l \in \{0, \ldots, p-1\}$ be the unique integer such that $a+l \equiv 0 \pmod{p}$. 
If $p$ is an odd prime, then we have 
\begin{align*}
\df_{a, \ldots, a}(t)=((-1)^{d+1}t)^l \df_{a, \ldots, a}(t^{-1})
\end{align*}
in $\Z_p\langle t, t^{-1}, h_{a, \ldots, a}(t)^{-1}\rangle$, where $\df_{a, \ldots, a}(t^{-1})$ is defined by $\iota (\df_{a, \ldots, a}(t))$.  
If $p=2$, then the above holds up to sign. 
\end{conj}
It is known that Conjecture \ref{conj} holds modulo $p$ (\cite[Proposition 4.11 (1)]{Wang}). 
Wang \cite[Theorem 4.6]{Wang2} proves that Conjecture \ref{conj} is true for $d=0$ by direct computations. 
Wang \cite[Theorem 4.20]{Wang} also proves that Conjecture \ref{conj} is true for $d=1$ under the following two conditions: 
\begin{enumerate}
\item $a \in N^{-1} \Z$ and $p > N$ for some $N \geq 2$,   
\item  $0 < a < 1$.   
\end{enumerate}
The purpose of this paper is to generalize this result as follows. 

\begin{thm} \label{main:1}
Suppose that the following two conditions: 
\begin{enumerate} 
\item $a \in N^{-1} \Z$ and $p \nmid N$ for some $N \geq 2$,   
\item  $0 < a < 1$.  
\end{enumerate}
Then, Conjecture \ref{conj} is true for general $d \geq 1$. 
Furthermore, if $p=2$,  
$$\df_{a, \ldots, a}(t)=
-((-1)^{d+1}t)^l \df_{a, \ldots, a}(t^{-1})
$$
holds 
if and only if $d$ is even and $a' \equiv 1 \pmod{2}$. 
\end{thm}
We give a sketch of the proof.  
The key tool of the proof is Theorem \ref{frob} (\cite[Theorem 4.6]{A}), which describes the eigenvalue of a Frobenius action on the de Rham cohomology group of the hypergeometric scheme 
$$(1-x_0^N) \cdots (1-x_d^N)=t$$
which is introduced by Asakura \cite{A2} (see Section \ref{HGsch}), in terms of Dwork's $p$-adic hypergeometric function. 
Let $W=W(\overline{\F}_p)$ be the Witt ring of $\overline{\F}_p$ and $K$ be the fractional field of $W$. 
For $t \in K$, $t \neq  0, 1$,  
let $V^*_t$ be the affine scheme over $K$ defined by 
$$w^N=z_1 \cdots z_d (1-z_1)^{N-1} \cdots (1-z_d)^{N-1}(z_1 \cdots z_d -t)^{N-1}, \quad w \neq 0$$
which is a quotient of the hypergeometric scheme 
(see Section \ref{proof}). 
Put $t=t_0^N$. 
The key idea of the proof is to construct an isomorphism 
\begin{align*}
&\iota \colon V^*_{t_0^N}  \to V^*_{t_0^{-N}} ; \quad (z_1, \ldots, z_d, w) \mapsto (u_1, \ldots, u_d, v) 
\end{align*}
given by 
\begin{align*}
u_i=z_i^{-1}  (i=1, \ldots, d), \quad  v=
\left\{
\begin{array}{ll}
\frac{\xi w}{z_1^2 \cdots z_d^2 t_0^{N-1}}  & (\textit{$N$ and $d$ are even}), \\
\frac{w}{z_1^2 \cdots z_d^2 t_0^{N-1}}& (otherwise),
\end{array}
\right.
\end{align*}
where $\xi \in K^*$ such that $\xi^N=-1$.  
We obtain the theorem by comparing the eigenvalues of a Frobenius action on the de Rham cohomology groups of $V^*_{t_0^N}$ and $V^*_{t_0^{-N}}$, where the former is described by $\df_{a, \ldots, a}(t_0^{N})$ and the latter is described by $\df_{a, \ldots, a}(t_0^{-N})$, respectively. 

This paper is constructed as follows. 
In Section \ref{HGsch}, we recall some properties of hypergeometric schemes according to \cite{A}. 
In Section \ref{proof}, we prove Theorem \ref{main:1}.  
In the appendix, we introduce a finite analogue of Theorem \ref{main:1}: the transformation formula of a hypergeometric function over a finite field between $t$ and $t^{-1}$, which implies the special case of Theorem \ref{main:1}.

\subsection{Notations}
Let $N$ be a positive integer and  $p$ be a prime such that $p \nmid N$. 
Let  $W=W(\Fp)$ denote the Witt ring of $\Fp$ and $K$ denote the fractional field of $W$. 
Put $A=W[t, (t-t^2)^{-1}]$ and $A_K = K \otimes_W A$. 
Let $\mu_N$ denote the group of $N$th roots of unity in $K$.    
For a smooth morphism $\pi \colon X \to \spec A_K$ of affine schemes over $K$, 
let $\mathscr{H}^i_{\rm dR}(X/A_K)$ denote the relative de Rham cohomology group $R^i \pi_* \Omega^{\bullet}_{X/A_K}$, and $H^i_{\rm dR}(X/A_K)$ denote the global section of $\mathscr{H}^i_{\rm dR}(X/A_K)$.  
For a finite abelian group $G$ and a character 
$\chi \in \operatorname{Hom}(G, K^*)$, 
 we write
$$e^{\chi}= \dfrac1{|G|} \sum_{g \in G} {\chi}(g^{-1}) g \in K[G]$$
for the corresponding projector. 
If $G$ acts on $\mathscr{H}^i_{\rm dR}(X/A_K)$, we write $\mathscr{H}^i_{\rm dR}(X/A_K)^{\chi}$ for the image of the projector $e_{\chi}$. 

\section{Hypergeometric scheme} \label{HGsch}
Let $f \colon U \to \spec A_K$ be a family of affine schemes over $K$, whose general fiber $U_t=f^{-1}(t)$ is defined by 
$$(1-x_0^N)\cdots (1-x_d^N)=t. $$
This is smooth over $A_K$.    
We call $U$ the hypergeometric scheme over $A_K$.  
The group $\mu_N^{d+1}$ acts on $U$ over $A_K$ by 
\begin{align*}
&(\nu_0, \ldots, \nu_{d})\cdot (x_0, \cdots, x_d)=(\nu_0 x_0, \ldots, \nu_d x_d),  \quad (\nu_0, \ldots, \nu_d) \in \mu_N^{d+1}. 
\end{align*}
This action induces an action on $H_{\rm dR}^d(U/A_K)$, 
hence $H_{\rm dR}^d(U/A_K)$ becomes an $A_K[\mu_N^{d+1}]$-module. 
Let $I_+=\{(i_0, \ldots, i_d) \in \Z^{d+1} \mid 0 < i_k < N\}$.  
In this section, we suppose that $(i_0, \ldots, i_d) \in I_+$. 
Put $\ua=(1-i_0/N, \ldots, 1-i_d/N)$ and $\check{\ua}=(i_0/N, \ldots, i_d/N)$.  
Let $\mathscr{D}=K\langle t, (t-t^2)^{-1}, \frac{d}{dt} \rangle$. 
For a character $\chi_{i_0, \ldots, i_d} \in \operatorname{Hom}(\mu_N^{d+1}, K^*)$ given by 
$$(\nu_0, \ldots, \nu_d) \mapsto  \nu_0^{i_0} \cdots \nu_d^{i_d}, $$
there is an isomorphism 
\begin{align*} 
\mathscr{D}/\mathscr{D} P_{{\rm HG}, \ua} \xrightarrow{\simeq} H^d_{\rm dR}(U/A_K)^{\chi_{i_0, \ldots, i_d}}; \quad P \mapsto P(\omega_{i_0, \ldots, i_d})
\end{align*}
as $\mathscr{D}$-modules (see \cite[Corollary 3.6]{A2}). 
Here, $\omega_{i_0, \ldots, i_d}$ is the  holomorphic $d$-form defined by 
$$\omega_{i_0, \ldots, i_d}=N^{-1} x_0^{i_0-N }x_1 ^{i_1-1} \cdots x_d^{i_d-1}\frac{dx_1 \wedge \cdots \wedge dx_d}{(1-x_1^N) \cdots (1-x_d^N)}. $$
We put $H^d_{{\rm dR}}(U/A_K)^{\chi_{i_0, \ldots, i_d}}_{K((t))}=K((t)) \otimes_{A_K} H^d_{{\rm dR}}(U/A_K)^{\chi_{i_0, \ldots, i_d}}$. 

\begin{prop}[{\cite[Proposition 4.1]{A}}]
Put $D=\frac{d}{dt}$. 
There is an element $\widehat{\eta}_{i_0, \ldots, i_d} \in H^d_{{\rm dR}}(U/A_K)^{\chi_{i_0, \ldots, i_d}}_{K((t))}$ such that 
$$\operatorname{Ker}\left[ D \colon H^d_{{\rm dR}}(U/A_K)^{\chi_{i_0, \ldots, i_d}}_{K((t))} \to H^d_{{\rm dR}}(U/A_K)^{\chi_{i_0, \ldots, i_d}}_{K((t))} \right] =K \widehat{\eta}_{i_0, \ldots, i_d}.$$
\end{prop}
For the polynomial $h_{\ua}(t)$ defined in \eqref{ha},  
we put  
$$h(t)=\prod_{(i_0, \ldots, i_d) \in I_+} h_{\ua}(t). $$
Put $B=A[h(t)^{-1}]$, $B_K=K \otimes_W B$, 
$$\hB=A[h(t)^{-1}]^{\wedge}:=\varprojlim_{n} \left(W/p^nW[t, (t-t^2)^{-1}, h(t)^{-1}] \right)$$
the $p$-adic completion and $\hB_K= K \otimes_W \hB$. 
By the congruence relation of Dwork's $p$-adic hypergeometric function, $\df_{\ua}(t)$ becomes an element of $\widehat{B}_K$, which is invertible (cf. \cite[Corollary 2.3]{A3}). 
Put 
$$H^d_{{\rm dR}}(U/A_K)^{\chi_{i_0, \ldots, i_d}}_{\hB_K}= \hB_K \otimes_{A_K} H^d_{{\rm dR}}(U/A_K)^{\chi_{i_0, \ldots, i_d}}. $$
We define  the unit root vector by 
$$\eta_{i_0, \ldots, i_d}= F_{\cua}(t)^{-1} \widehat{\eta}_{i_0, \ldots, i_d}, $$
which defines an element of $H^d_{{\rm dR}}(U/A_K)^{\chi_{i_0, \ldots, i_d}}_{\hB_K}$ (see \cite[Lemma 4.2]{A}).  
Put 
$$H^d_{{\rm dR}}(U/A_K)^{\chi_{i_0, \ldots, i_d}}_{\hB_K, \text{unit}} = \hB_K \eta_{i_0, \ldots, i_d}. $$
There is a perfect pairing 
\begin{align*} 
Q \colon H^d_{{\rm dR}}(U/A_K)^{\chi_{i_0, \ldots, i_d}}_{\hB_K} \otimes_{\hB_K} H_{\rm dR} (U / A_K)^{\chi_{-i_0, \ldots, -i_d}}_{\hB_K} \to \hB_K
\end{align*} 
(see \cite[Section 3.2 (3.5)]{A}). 
Let 
$$VH^d_{{\rm dR}}(U/A_K)^{\chi_{i_0, \ldots, i_d}}_{\hB_K} \subset H^d_{{\rm dR}}(U/A_K)^{\chi_{i_0, \ldots, i_d}}_{\hB_K}$$
be a $\hB_K$-submodule which is the exact annihilator of $H^d_{{\rm dR}}(U/A_K)^{\chi_{-i_0, \ldots, -i_d}}_{\hB_K, \text{unit}} $. 
Then we have 
\begin{align}
H^d_{{\rm dR}}(U/A_K)^{\chi_{i_0, \ldots, i_d}}_{\hB_K}/VH^d_{{\rm dR}}(U/A_K)^{\chi_{i_0, \ldots, i_d}}_{\hB_K} \simeq \hB_K  \omega_{i_0, \ldots, i_d} \label{VH}
\end{align}
(see \cite[Lemma 4.3]{A}).

Let $A_K^{\dagger}$ be the weak completion of $A_K$ (cf. \cite[p.135]{LS}).  
Let $\sigma$ be the $F$-linear Frobenius on $A_K^{\dagger}$ given by $\sigma(t)=t^p$, where $F$ is the $p$th Frobenius on $W$. 
Then there is a $p$th Frobenius $\Phi$ induced by $\sigma$ acting on 
$$H^d_{{\rm dR}}(U/A_K)^{\chi_{i_0, \ldots, i_d}}_{A_K^{\dagger}}:=A_K^{\dagger}\otimes_{A_K} H^d_{{\rm dR}}(U/A_K)^{\chi_{i_0, \ldots, i_d}}. $$ 
The Frobenius $\sigma$ extends to $\hB_K$, hence we can extend $\Phi$ to that on $H^d_{{\rm dR}}(U/A_K)^{\chi_{i_0, \ldots, i_d}}_{\hB_K}$ and denote it by the same notation.

\begin{thm}[{\cite[Theorem 4.6]{A}}] \label{frob}
For $(i_0, \ldots, i_d), (j_0, \ldots, j_d) \in I_+$, we suppose that $p j_k \equiv i_k  \pmod{N}$ for $k=0, \ldots, d$. 
Then we have 
$$\Phi(\omega_{j_0, \ldots, j_d}) \equiv p^d \df_{\ua} (t)^{-1} \omega_{i_0, \ldots, i_d} \mod{VH^d_{{\rm dR}}(U/A_K)^{\chi_{i_0, \ldots, i_d}}_{\hB_K}}, $$
where $\ua=(1-i_0/N, \ldots, 1-i_d/N)$. 
\end{thm}

\section{Proof of Theorem \ref{main:1}} \label{proof}
Let $g \colon V^* \to \spec A_K$ be a family of affine schemes over $K$, whose general fiber $V^*_t=g^{-1}(t)$ is defined by 
$$w^N=z_1 \cdots z_d (1-z_1)^{N-1} \cdots (1-z_d)^{N-1}(z_1 \cdots z_d -t)^{N-1}, \quad w \neq 0,  $$
which is smooth over $A_K$. 
Let $U$ be the hypergeometric scheme defined in Section \ref{HGsch} and $U^*$ be the open subscheme of $U$ defined by $x_0 \cdots x_d \neq 0$. 
Then there is a covering map 
\begin{align*} 
&\rho \colon U^* \to V^*; \quad \left\{
\begin{array}{ll}
w=x_0^{N-1} \cdots x_d^{N-1} (1-x_1^N) \cdots (1-x_d^N), \\
z_i=1-x_i^N  \quad (i=1, \ldots, d). 
\end{array}
\right.
\end{align*}     
The group $\mu_N$ acts on $V^*$ over $A_K$ by 
$$\nu \cdot (z_1, \ldots, z_d, w)=(z_1, \ldots, z_d, \nu^{-1} w), \quad \nu \in \mu_N. $$
The action induces an action on $H^d_{\rm dR}(V^*/A_K)$, hence $H^d_{\rm dR}(V^*/A_K)$ becomes an $A_K[\mu_N]$-module. 
For $n \in \Z$, let 
$$\chi_n \colon \mu_N \to K^*; \quad \nu \mapsto \nu^n$$
be a character of $\mu_N$. 
Then we have the following proposition. 
\begin{prop}
If $n \not \equiv 0 \pmod{N}$, then we have an isomorphism 
$$H^d_{\rm dR}(V^*/A_K)^{\chi_{n}} \simeq  H^d_{\rm dR}(U/A_K)^{\chi_{n, \ldots, n}}$$
as $\mathscr{D}$-modules.  
\end{prop}
\begin{proof}
Since $U^*$ is Galois over $V^*$ with Galois group $\operatorname{Ker}(m)$, where $m \colon \mu_N^{d+1} \to \mu_N$ is the multiplication, 
we have 
\begin{align*}
H^d_{\rm dR}(U^*/A_K)^{\chi_{n, \ldots, n}} \simeq H^d_{\rm dR}(V^*/A_K)^{\chi_n}. 
\end{align*}
Put $D=U \backslash U^*$. 
We have a localization exact sequence 
\begin{align*}
\cdots &\to \mathscr{H}^{d-2}_{\rm dR}(D/A_K) \to \mathscr{H}^d_{\rm dR}(U/A_K) \to \mathscr{H}^d_{\rm dR}(U^*/A_K) \to \mathscr{H}^{d-1}_{\rm dR}(D/A_K) \to \cdots
\end{align*}
(cf. \cite[Theorem (3.3)]{Hartshorne}). 
Since each irreducible component of $D$ is given by $\{x_i=0\}$ and $i$th component of $\mu_N^{d+1}$ acts trivially on $\{x_i=0\}$, we have 
$$\mathscr{H}^j_{\rm dR}(D/A_K)^{\chi_{n, \ldots, n}}=0$$
by the assumption $n \not \equiv 0 \pmod{N}$.  
Therefore, we have an isomorphism 
$$\mathscr{H}^d_{\rm dR}(U/A_K)^{\chi_{n, \ldots, n}} \simeq \mathscr{H}^d_{\rm dR}(U^*/A_K)^{\chi_{n, \ldots, n}}, $$
which induces an isomorphism 
$$H^d_{\rm dR}(U/A_K)^{\chi_{n, \ldots, n}} \simeq H^d_{\rm dR}(U^*/A_K)^{\chi_{n, \ldots, n}}. $$
\end{proof}
Let 
$\omega_n$ be the holomorphic $d$-form on $V^*$ defined by  
$$\omega_n=(1-z_1)^{n-1} \cdots (1-z_n)^{n-1}(z_1 \cdots z_d -t)^{n-1}\frac{dz_1 \wedge \cdots \wedge dz_d}{w^n},  $$
which defines an element of $H^d_{\rm dR}(V^*/A_K)^{\chi_n}$. 
Then we compute that 
\begin{align} \label{rho}
\begin{split}
\rho^*(\omega_n) & =(-1)^d N^d x_0^{n-N }x_1 ^{n-1} \cdots x_d^{n-1}\frac{dx_1 \wedge \cdots \wedge dx_d}{(1-x_1^N) \cdots (1-x_d^N)} \\
&=(-1)^d N^{d+1} \omega_{n, \ldots, n}, 
\end{split}
\end{align}
hence $H^d_{\rm dR}(V^*/A_K)^{\chi_n}$ is a $\mathscr{D}$-module generated by $\omega_n$.  

Let $\widehat{f} \colon \widehat{U} \to \spec A_K$ (resp. $\widehat{g} \colon \widehat{V}^* \to \spec A_K$) be a family of affine schemes over $K$, whose general fiber is given by $\widehat{f}^{-1}(t)=U_{t^{-1}}$ (resp. $\widehat{g}^{-1}(t)=V^*_{t^{-1}}$).

Put $t=t_0^N$. 
From now on, we change the notations and let 
$A=W[t_0, t_0^{-1}, (1-t_0^N)^{-1}]$, $A_K=K \otimes_W A$ and $\widehat{B}_K=K \otimes_W A[h(t_0^N)^{-1}]^{\wedge}$, where $h(t)$ is the polynomial defined in Section \ref{HGsch}.  
Then there is an isomorphism 
\begin{align} \label{isom}
&\iota \colon V^*  \to \wV^* ; \quad (z_1, \ldots, z_d, w) \mapsto (u_1, \ldots, u_d, v) 
\end{align}
over $A_K$, which is given by 
\begin{align*}
u_i=z_i^{-1}  (i=1, \ldots, d), \quad  v=
\left\{
\begin{array}{ll}
\frac{\xi w}{z_1^2 \cdots z_d^2 t_0^{N-1}}  & (\textit{$N$ and $d$ are even}), \\
\frac{w}{z_1^2 \cdots z_d^2 t_0^{N-1}}& (otherwise), 
\end{array}
\right.
\end{align*} 
where $\xi \in K^*$ such that $\xi^N=-1$.  
For $n \in \{1, \ldots, N-1\}$, define  
$$\varphi_{N, d}(n)=
\left\{
\begin{array}{ll}
{n}  & (\textit{$N$ and $d$ are even}), \\ 
0 & (otherwise). 
\end{array}
\right.
$$
Since we have 
\begin{align}  \label{iota2}
\iota^*(\omega_{n})=(-1)^{(d+1)n-1}\xi^{-\varphi_{N, d}(n)}t_0^{N-n} \omega_{n}, 
 \end{align}
the map \eqref{isom} induces an isomorphism 
$$\iota^* \colon H_{\rm dR}^d(\wV^* /A_K)^{\chi_n} \xrightarrow{\simeq} H_{\rm dR}^d(V^*/A_K)^{\chi_n}. $$
We define a morphism 
$$\widetilde{\iota^*} \colon H^d_{\rm dR}(\wU/A_K)^{\chi_{n, \ldots, n}} \to H^d_{\rm dR}(U/A_K)^{\chi_{n, \ldots, n}}$$ 
so that commutes the following diagram 
$$
  \begin{CD}
     H_{\rm dR}^d(\wV^* /A_K)^{\chi_n}  @>{\simeq}>>  H^d_{\rm dR}(\wU/A_K)^{\chi_{n, \ldots, n}} \\
  @V{\iota^*}V{\simeq}V    @VV{\widetilde{\iota^*}}V \\
  H_{\rm dR}^d(V^* /A_K)^{\chi_n}   @>{\simeq}>> H^d_{\rm dR}(U/A_K)^{\chi_{n, \ldots, n}} .  
  \end{CD}
$$
Then by \eqref{rho} and \eqref{iota2}, we compute explicitly 
\begin{align} \label{action}
 \widetilde{\iota^*}(\omega_{n, \ldots, n})=(-1)^{(d+1)n-1}\xi^{-\varphi_{N, d}(n)}t_0^{N-n} \omega_{n, \ldots, n}, 
 \end{align}
 hence by \eqref{VH}, we have 
 \begin{align} \label{iw}
 \widetilde{\iota^*}(VH^d_{\rm dR}(\wU/A_K)^{\chi_{n, \ldots, n}}_{\hB_K}) \simeq VH^d_{\rm dR}(U/A_K)^{\chi_{n, \ldots, n}} _{\hB_K}. 
\end{align}

\begin{proof}[Proof of Theorem \ref{main:1}]
Let $\sigma$ be the $F$-linear $p$th Frobenius on $\widehat{B}_K$ given by $\sigma(t_0)=t_0^p$, and $\Phi$ be the $p$th Frobenius induced by $\sigma$ acting on the de Rham cohomology groups.  
Let $m, n \in \{1, \ldots, N-1\}$ be integers such that $pm \equiv n \pmod{N}$. 
By Theorem \ref{frob} and \cite[Proposition 4.11 (3)]{Wang}, we have 
\begin{align}
\Phi(\omega_{m, \ldots, m}) \equiv p^d\df_{\ua}(t_0^{-N})^{-1} \omega_{n \ldots, n}  \mod{VH^d_{\rm dR}(\wU/A_K)^{\chi_{n, \ldots, n}} _{\hB_K}},   \label{phi}
\end{align}
where $\ua=(1-n/N, \ldots, 1-n/N)$. 
Apply $\widetilde{\iota^*}$ on the both sides of \eqref{phi}. 
By \eqref{action} and \eqref{iw}, we have 
\begin{align*}
&\Phi((-1)^{(d+1)m-1}\xi^{-\varphi_{N, d}(m)}t_0^{N-m}\omega_{m, \ldots, m}) \\
& \equiv p^d\df_{\ua}(t_0^{-N})^{-1} ((-1)^{(d+1)n-1}\xi^{-\varphi_{N, d}(n)}t_0^{N-n} \omega_{n, \ldots, n})  \mod{VH^d_{\rm dR}(U/A_K)^{\chi_{n, \ldots, n}} _{\hB_K}}. 
\end{align*}
By Theorem \ref{frob}, the left-hand side is computed as 
\begin{align*}
& (-1)^{(d+1)m-1}\xi^{-p\varphi_{N, d}(m)}t_0^{p(N-m)} p^d \df_{\ua}(t_0^N)^{-1}\omega_{n, \ldots, n} \mod{VH^d_{\rm dR}(U/A_K)^{\chi_{n, \ldots, n}} _{\hB_K}}, 
\end{align*}
hence we obtain  
\begin{align*}
\df_{\ua}(t)&=(-1)^{(d+1)(m-n)}\xi^{-(p\varphi_{N, d}(m) - \varphi_{N, d}(n))} t^{pa_m-a_n}\df_{\ua}(t^{-1}) \\ 
&=(-1)^{(d+1)(m-n)-\frac{p\varphi_{N, d}(m) - \varphi_{N, d}(n)}{N}} t^{pa_m-a_n}\df_{\ua}(t^{-1}), 
\end{align*}
where we put $a_i=1-i/N$. 

First, we prove the case $p \neq 2$. 
Since we have 
$$pa_m -a_n=p-1-\frac{pm-n}{N}=l, $$ 
it suffices to show that 
\begin{align} 
(d+1)(m-n) - \dfrac{p\varphi_{N, d}(m) - \varphi_{N, d}(n)}{N} \equiv 
 (d+1)l \pmod{2}.  \label{sign}
\end{align}

Suppose that $N$ and $d$ are even. 
Then we compute that 
\begin{align*}
&(d+1)(m-n)-\dfrac{p\varphi_{N, d}(m) - \varphi_{N, d}(n)}{N} \\ 
&\qquad = (d+1)(m-n)-\frac{pm-n}{N}  \\
&\qquad \equiv N \dfrac{pm-n}{N} -\frac{pm-n}{N} \\
&\qquad \equiv -\frac{pm-n}{N} \equiv l \equiv (d+1)l \pmod{2}, 
\end{align*}
hence \eqref{sign} holds.  

Next, suppose that $N$ is odd or $d$ is odd. 
Then the left-hand side of \eqref{sign} is 
\begin{align*}
(d+1)(m-n)-\dfrac{p\varphi_{N, d}(m) - \varphi_{N, d}(n)}{N} = (d+1)(m-n). 
\end{align*}
If $d$ is odd, then \eqref{sign} holds. 
If $N$ is odd and $d$ is even, then we have 
\begin{align*}  
l=p-1-\frac{pm-n}{N} 
 \equiv m-n \pmod{2}, 
 \end{align*}
hence \eqref{sign} holds.  

Secondly, we prove the case $p=2$. Since $N$ is odd, it suffices to show that 
\begin{align} \label{2}
(d+1)(m-n)\equiv   \left\{
\begin{array}{ll}
1+(d+1)l & (\textit{$d$ is even and $(a_n)'\equiv 1 \pmod{2}$}),\\
(d+1)l  & (otherwise)
\end{array}
\right.
\end{align}
modulo $2$. 
If $d$ is odd, then \eqref{2} holds. 
Suppose that $d$ is even.  
Then we have 
$$(d+1)(m-n) \equiv m-n \pmod{2}$$
and 
$$(d+1)l \equiv l \equiv 1-\dfrac{2m-n}{N} \equiv 1-n \pmod{2}. $$
Therefore, 
$$(d+1)(m-n) \equiv 1+(d+1)l \pmod{2}$$
holds if and only if $m \equiv 0 \pmod{2}$, which means that $(a_n)'=a_m=1-m/N \equiv 1 \pmod{2}$. 
Hence, \eqref{2} holds. 
\end{proof}

For a positive integer $f$, put $q=p^f$ and define 
$$\mathscr{F}^{{\rm Dw}, f}_{\ua}(t)= \prod_{i=0}^f \mathscr{F}^{{\rm Dw}}_{\ua^{(i)}}(t^{p^i})=
\dfrac{F_{\ua}(t)}{F_{\ua^{(f)}}(t^q)} \in  \Z_p\langle t, h_{\ua}(t)^{-1}\rangle. $$
Repeatedly using Theorem \ref{main:1}, we have the following corollary. 
\begin{cor} \label{maincor}
Suppose the conditions in Theorem \ref{main:1} (i) and (ii).  
Let $l \in \{1, \ldots, q-1\}$ be the unique integer such that $a + l \equiv 0 \pmod{q}$. 
If $p \neq 2$, then we have  
$$\mathscr{F}^{{\rm Dw}, f}_{a, \ldots, a}(t)=
((-1)^{d+1}t)^l \mathscr{F}^{{\rm Dw}, f}_{a, \ldots, a}(t^{-1}) 
$$
in $\Z_p\langle t, t^{-1}, h_{a, \ldots, a}(t)^{-1}\rangle$. 
If $p=2$, then the above holds up to sign. 
\end{cor}

\appendix
\section{Transformation formula of hypergeometric functions over finite fields} \label{FHG}
Here we introduce a transformation formula of a hypergeometric function over a finite field, which is a finite analogue of Theorem \ref{main:1}. 
Over the finite fields, there are several definitions of hypergeometric functions due to Koblitz \cite{Koblitz}, Greene \cite{Greene}, McCarthy \cite{McCarthy}, Fuselier-Long-Ramakrishna-Swisher-Tu \cite{FLRST}, Katz \cite{Katz}, and Otsubo \cite{Otsubo}.  
In this paper, we use Otsubo's definition. 
For a positive integer $n$, let $\mu_n \subset \C^*$ denote the group of $n$th roots of unity. 
Let $\F_q$ be a finite field of characteristic $p$ with $q=p^f$ elements. 
Put $\widehat{\F_q^{*}}=\operatorname{Hom}(\F_q^*, \C^*)$ and for any $\varphi \in \widehat{\F_q^{*}}$, we set $\varphi(0)=0$ and write $\overline{\varphi}=\varphi^{-1}$.  
Let $\varepsilon \in \widehat{\F_q^{*}}$ be the trivial character, i.e. for any $x \in \F_q$, 
$$\varepsilon(x)=
\left\{
\begin{array}{ll}
1 & (x \neq 0),  \\
0 & (x=0). 
\end{array}
\right. $$
We fix a non-trivial additive character $\psi \in \operatorname{Hom}(\F_q, \C^*)$.  
For a multiplicative character $\vp \in \widehat{\F_q^*}$, the Gauss sum, which is a finite analogue of the gamma function, and its variant are defined by 
$$g(\varphi)=-\sum_{x \in \F_q}\varphi(x) \psi(x), \quad g^{\circ}(\varphi)= \left\{
\begin{array}{ll}
g(\varphi) & (\varphi \neq \varepsilon),  \\
q & (\varphi = \varepsilon) 
\end{array}
\right. 
\quad \in \Q(\mu_{p(q-1)}). $$
Then we define the Pochhammer symbol and its variant by 
$$(\varphi)_{\nu}=\frac{g(\vp \nu)}{g(\vp)}, \quad (\varphi)^{\circ}_{\nu}=\frac{g^{\circ}(\vp \nu)}{g^{\circ}(\vp)}, $$
where $\vp, \nu \in \widehat{\F_q^*}$. 
For any $\varphi \in \widehat{\F_q^*}$, we have 
\begin{equation}  
g(\varphi) g^{\circ}(\overline{\varphi})=\varphi(-1) q.  \label{gs}
\end{equation}

With the notations as above, for $\alpha_i, \beta_j, \in \widehat{\F_q^*}$, we define the hypergeometric function over the finite field $\F_q$ by  
\begin{align*} 
{_{d+1}F_{d}}\left( 
\begin{matrix}
\alpha_0, \ldots, \alpha_{d} \\
\beta_1, \ldots, \beta_d
\end{matrix}
; t
\right)
= \dfrac1{1-q}\sum_{\nu \in \widehat{\F_v}} \dfrac{(\alpha_0)_{\nu} \cdots (\alpha_{d})_{\nu} }{(\varepsilon)^{\circ}_{\nu} (\beta_1)^{\circ}_{\nu} \cdots (\beta_d)^{\circ}_{\nu}  } \nu(t) \in \Q(\mu_{q-1}). 
\end{align*}

The transformation formula of the hypergeometric function over the finite field between $t$ and $t^{-1}$ is as follows. 

\begin{prop} \label{ft}
For $\alpha_i$, $\beta_j \in \widehat{\F_q^*}$ $(1 \leq i, j \leq d)$, suppose that $\alpha_i \neq \beta_i$ for all $i$. 
For any $\al_0 \in \widehat{\F_q^*}$, $\al_0 \neq \e$, and $t \in \F_q^*$, we have
$${_{d+1}F_{d}}\left( 
\begin{matrix}
\al_0, \cdots, \al_d \\
\beta_1, \ldots \beta_d
\end{matrix}
; t
\right)=
\overline{\al_0}(-t)
\prod_{i=1}^d \dfrac{g(\al_0 \overline{\beta_i}) g^{\circ} (\overline{\al_i})}{g^{\circ}(\al_0 \overline{\al_i}) g(\overline{\beta_i})}
{_{d+1}F_{d}}\left( 
\begin{matrix}
\al_0, \al_0 \overline{\beta_1}, \cdots, \al_0  \overline{\beta_d} \\
\al_0 \overline{\al_1}, \al_0 \overline{\al_2}, \ldots \al_0 \overline{\al_d}
\end{matrix}
; t^{-1}
\right). 
$$
In particular, if $\al_0 = \cdots =\al _d =\al \neq \e$ and $\beta_1 = \cdots =\beta_d = \e$, then we have 
$${_{d+1}F_{d}}\left( 
\begin{matrix}
\al, \cdots, \al \\
\e, \ldots \e
\end{matrix}
; t
\right)
= \overline{\al}((-1)^{d+1}t) 
{_{d+1}F_{d}}\left( 
\begin{matrix}
\al, \cdots, \al \\
\e, \ldots \e
\end{matrix}
; t^{-1}
\right), $$
which is a finite analogue of Theorem \ref{main:1}. 
\end{prop}

\begin{proof}
The first assertion follows from \cite[Propositions 2.10 and 2.11]{Otsubo} and the second assertion follows from \eqref{gs}. 
\end{proof}

\begin{rmk}
For $t \in \F_q \setminus \{0, 1\}$ and a positive integer $N \geq 2$ such that $N \mid q-1$, let $X_t$ be the smooth projective curve over $\F_q$ birational to the affine curve 
$$(1-x_0^N)(1-x_1^N)=t. $$
It is known that the Frobenius trace of $X_t$, i.e. the unit root of the Artin $L$-function of $X_t$ is written in terms of hypergeometric functions over $\F_q$ (see \cite[Theorem 4.2]{AO}). 
On the other hand, the unit root of that is written in terms of $\mathscr{F}^{{\rm Dw}, f}_{\ua}(t)$ (see \cite[Corollary 4.7]{A}).   
Therefore, by using Proposition \ref{ft} and comparing the unit root of the Artin $L$-function of $X_t$ with that of $X_{t^{-1}}$, we can prove the special case of Corollary \ref{maincor}: 
Let $W=W(\F_q)$ be the Witt ring of $\F_q$ and $\mu_{q-1}$ be the group of $(q-1)$th roots of unity in $W$. 
For $a \in \{1, \ldots, N-1\}$ and $\zeta \in \mu_{q-1}$ such that $h_{\frac{a}N, \frac{a}N}(\zeta) \not \equiv 0 \pmod{pW}$, we have 
$$\mathscr{F}^{{\rm Dw}, f}_{\frac{a}N, \frac{a}{N}}(\zeta)=
\zeta^l \mathscr{F}^{{\rm Dw}, f}_{\frac{a}N, \frac{a}{N}}(\zeta^{-1}),  $$ 
where $l \in \{0, \ldots, q-1\}$ is the unique integer such that $a/N + l \equiv 0 \pmod{q}$. 
Here, the special value of Dwork's $p$-adic hypergeometric function at $t=t_0$ is defined by 
$$\mathscr{F}^{{\rm Dw}, f}_{a_0, a_1}(t_0)=\mathscr{F}^{{\rm Dw}, f}_{a_0, a_1}(t)|_{t=t_0} = \lim_{n \to \infty} \left.  \dfrac{[F_{a_0, a_1}(t)]_{<p^n} }{[F_{a^{(f)}_0, a^{(f)}_1}(t^q)]_{<p^n}} \right|_{t=t_0}.  $$
\end{rmk}

\section*{Acknowledgment}
The author would like to sincerely thank Yoshinosuke Hirakawa and Noriyuki Otsubo for their valuable and helpful comments.  
This paper is a part of the outcome of research performed under Waseda University Grant for Special Research Projects (Project number: 2024C-280) and Kakenhi Applicants (Project number: 2024R-054).

\end{document}